\documentclass[12pt, twoside]{article}
\usepackage{graphicx}
\usepackage{tikz}
\usepackage{authblk}
\usepackage{amssymb,amsmath,amsthm,latexsym}
\usepackage{amsfonts}
\usepackage{enumerate}
\newtheorem{thm}{Theorem}[section]

\newtheorem{cor}[thm] {Corollary}
\newtheorem{lem} [thm]{Lemma}

\theoremstyle{definition} 

\newtheorem{ex}[thm]{Example}
\newtheorem{rmk}[thm] {Remark}
\newtheorem{defn}[thm]{Definition}
\numberwithin{equation}{section}

\newcommand\mdd{\operatorname{mdd}}
\newcommand\ext{\operatorname{ex}}
\newcommand\ter{\operatorname{ter}}

\begin{document}
	
		\title
		{
			\textbf{On the Metric Dimension\\  of Signed Graphs
				}
		}
			
\author[1]{ Shahul Hameed K}
\author[2]{Remna K P}
\author[2]{Divya T}
\author[3] {Biju K}
\author[2]{\\ Rajeevan P}
\author[2] {Santhosh G O}
\author[1] {Ramakrishnan K O}

\affil[1]{\footnotesize Department of
				Mathematics, K M M Government\ Women's\ College, Kannur - 670004, Kerala, India.  E-mail: shabrennen@gmail.com}
\affil[2]{\footnotesize Department of
			Mathematics, Government\ Brennen\ College, Kannur - 670106,\ Kerala,  \ India.}
\affil[3]{\footnotesize Department of
			Mathematics, P R N S S\ College, Mattannur, Kannur - 670702,\ Kerala,  \ India.}

\maketitle
\thispagestyle{empty}
\begin{abstract}
A signed graph $\Sigma$ is a pair $(G,\sigma)$, where $G=(V,E)$ is the underlying graph in which each edge is assigned $+1$ or $-1$ by the signature function $\sigma:E\rightarrow\{-1,+1\}$. In this paper, we extend the extensively applied concepts of metric dimension and resolving sets for unsigned graphs to signed graphs. We analyze the metric dimension of some well known classes of signed graphs including a special case of signed trees. Among other things, we establish that the metric dimension of a signed graph is invariant under negation.
\end{abstract}


\section{Introduction}
A signed graph $\Sigma=(G,\sigma)$ is a graph $G=(V,E)$ together with a signature function $\sigma:E\rightarrow \{1,-1\}$. In this paper, we extend the definition of resolving sets and metric dimension of unsigned graphs to signed graphs and compute metric dimensions of some classes of graphs and a special class of signed trees. The signed distance concept for signed graphs introduced in \cite{sdist} by Shahul Hameed K et al., will be utilized for this purpose.

All graphs in this paper are simple, finite and connected. The distance between two vertices $u$ and $v$ in the underlying graph $G$ will be usually denoted by $d(u,v)$ which is defined as the length of a shortest path between them. If $\Sigma=(G,\sigma)$ is a signed graph, we denote the sign of a path $P$ in $\Sigma$ by $\sigma(P)=\prod_{e\in P}\sigma(e)$. One of the shortest paths from a vertex $u$ to a vertex $v$ is denoted by $P(u,v)$ and collection of all such paths is denoted by $\mathcal{P}(u,v)$. There may be more than one paths between two vertices in a non-geodetic graph and as such it is defined \cite{sdist}  $\sigma_{\max}(uv)$ as $-1$ if $\sigma (P(u,v))=-1\  \forall \ P(u,v)\in \mathcal{P}(u,v)$ and $+1$, otherwise. Similarly $\sigma_{\min}(uv)$ as $+1$ if $\sigma (P(u,v))=+ 1 \ \forall  \  P(u,v) \in \mathcal{P}(u,v)$ and $-1$, otherwise. A signed graph is said to be homogeneous if all the edges have the same sign. Otherwise it is called non-homogeneous.

Given a connected signed graph $\Sigma=(G,\sigma)$ with $G=(V,E)$, there are two types of signed distances $d^{\max}$ and $d^{\min}$~\cite{sdist}, defined by (i) $d^{\max}(u,v)= \sigma_{\max}(uv)d(u,v)$ and (ii) $d^{\min}(u,v)= \sigma_{\min}(uv)d(u,v)$ for all $u,v\in V$.	 Two vertices $u$ and $v$ in a connected signed graph $\Sigma$ are said to be (distance) compatible~\cite{sdist}, if $d^{\max}(u,v)=d^{\min}(u,v)$. $\Sigma$ itself is said to be distance compatible or simply compatible~\cite{sdist}, if every pair of its vertices is distance compatible. In the case of a compatible signed graph $\Sigma$, we denote the common value of $d^{\max}(u,v)$ and $ d^{\min}(u,v)$ simply by $d_{\Sigma}(u,v)$ or, if there is no scope for confusion, by $d(u,v)$, as in the case of ordinary graphs. Also, in the case of a connected compatible signed graph $\Sigma$, we denote by $\sigma(uv)$, not only the edge sign but also the common value of $\sigma_{\max}$ and $\sigma_{\min}$.

The sign of a cycle $C$, denoted by $\sigma(C)$, in a signed graph $\Sigma=(G,\sigma)$ is the product of its edge signs. i.e., $\sigma(C)=\prod_{e\in E(C)}\sigma(e)$. A cycle $C$ in a signed graph is said to be positive if $\sigma(C)=1$. A signed graph is defined to be balanced if every cycle in it is positive. Given a signed graph $\Sigma=(G,\sigma)$, its negation is the signed graph $-\Sigma=(G,-\sigma)$. A signed graph $\Sigma$ is said to be anti-balanced if its negation $-\Sigma$ is balanced. Geodetic graphs are those graphs in which every pair of vertices will be joined by a unique shortest path. Balanced signed graphs, anti-balanced signed graphs and geodetic signed graphs are all examples of compatible signed graphs \cite{sdist}. A characterization of compatible signed graphs is given in \cite{sdist1}, as follows.
\begin{thm}[\cite{sdist1}]\label{compcycle}
A $2$-connected, non-geodetic signed graph is compatible if and only if it has no even negatively signed cycle $C_{2k}$ in which there are two diametrically opposite vertices with distance $k$.
\end{thm}
Resolving sets and basis for unsigned graphs were first introduced in the $1970$s by Slater \cite{slater} and independently by Harary and Melter \cite{harary}. There are many physical applications for these concepts on which extensive research is still on.  The reader may refer to \cite{r1, chartrand, harary, haup, kuller, slater, slater1, shan} for more details. In order to extend these concepts to the case of distance compatible signed graphs, first of all, let us have the following definitions.

Given a compatible signed graph $\Sigma=(V,E,\sigma)$, with respect to an ordered subset $W=\{w_1,w_2,\cdots w_k\}\subset V$of cardinality $k$, every vertex of $\Sigma$ has a $k$-vector representation as $r_{\Sigma}(v|W)=(d_\Sigma(v,w_1),d_\Sigma(v,w_2),\cdots d_\Sigma(v,w_k) )$ called the metric representation of the vertex $v$. If there is no scope for confusion, $r_{\Sigma}(v|W)$ will be simply denoted by $r(v|W)$. An ordered subset $W$ of the vertex set $V$ of a compatible signed graph $\Sigma=(V,E,\sigma)$ is called a resolving set if $r(v_i|W)\neq r(v_j|W)$ for all vertices $v_i\neq v_j$ in $V$.		
A resolving set $W$ of a compatible signed graph $\Sigma$ is called its basis if $W$ is a minimum resolving set. The cardinality of a basis is called the metric dimension of that signed graph, denoted by $\dim(\Sigma)$.

Regarding the organization of the paper, Section~\ref{s1} deals with some basic but important theorems regarding the resolving sets and the metric dimension of signed graphs including the metric dimension of such signed graphs like signed paths, signed cycles and signed stars. Section~\ref{s2} and Section~\ref{s3}, respectively, dwell upon exclusively with the metric dimension of signed wheels and signed trees. Unless otherwise mentioned all the signed graphs under consideration are distance compatible.
\section{Main Results}\label{s1}
\begin{thm}\label{undres} Every resolving set of the underlying graph $G$ is a resolving set for the signed graph $\Sigma=(G,\sigma)$.
\end{thm}
\begin{proof} Let $V=\{v_1, v_2, \dots , v_n\}$ be the vertex set of $G$ and $W=\{v_{i_1}, v_{i_2}, \dots, v_{i_k}\}$ be the resolving set of $G$.  Then $r(v_p|W)\neq r(v_q|W) \ \ \forall  \ v_p\neq v_q \in V$. If possible assume that $W$ is not a resolving set of $\Sigma$.  Then $r_\Sigma(v_s|W)=r_\Sigma(v_t|W)$ for some $v_s, v_t\in V$.  i.e., $(d_\Sigma(v_s, v_{i_1}), \dots, d_\Sigma(v_s, v_{i_k}))=(d_\Sigma(v_t, v_{i_1}), \dots, d_\Sigma(v_t, v_{i_k}))$. This gives, $d_\Sigma(v_s, v_{i_x})= d_\Sigma(v_t, v_{i_x})\ {\text{for}}\ x=1,2,\dots, k$. Hence, $|d_\Sigma(v_s, v_{i_x})|= |d_\Sigma(v_t, v_{i_x})| \ {\text{for}}\ x=1,2,\dots, k$. So, $d_G(v_s, v_{i_x})= d_G(v_t, v_{i_x}) \ {\text{for}}\ x=1,2,\dots, k$ making $r_G(v_s|W) = r_G(v_t|W)  \ \text{for some} \ v_s, v_t \in V(G)$, leading to a contradiction to the assumption.\ \ Hence, $W$ is a resolving set for the signed graph $\Sigma$ also.
\end{proof}
\begin{thm}\label{gandsigma}
For a signed graph $\Sigma=(G,\sigma)$, $1\le\dim(\Sigma)\le \dim(G)\le n-1$.
\end{thm}
\begin{proof}
The proof follows directly from the definitions since the co-ordinates for the metric representation of a vertex can include negative integers also in the case of a signed graph whereas that of the underlying graph allow only positive integers.
\end{proof}
In view of the above theorem, we define the metric dimensional difference of a signed graph $\Sigma=(G,\sigma)$, denoted by $\mdd(\Sigma)$, as $\mdd(\Sigma)=\dim(G)-\dim(\Sigma)$. Thus $\mdd(\Sigma)\ge 0$.
The following theorem says that negating a signed graph does not change the dimension.
\begin{thm}\label{negate}
	If $\Sigma$ is a compatible signed graph, then $\dim(\Sigma)=\dim(-\Sigma)$
\end{thm}
\begin{proof} Let $\Sigma=(G,\sigma)$ be the given signed graph and let $-\Sigma=(G,\sigma')$ be its negation where $\sigma'=-\sigma$. The sign of a shortest path $P(u,v)$ has the following relation under these two signatures.
	\begin{equation}\label{eq1}
	\sigma'(P(u,v))=(-1)^{d(u,v)}\sigma(P(u,v))
	\end{equation}
Let $\dim(\Sigma)=k$ and $W=\{w_1,w_2,\cdots,w_k\}$ be a basis for $\Sigma$. We claim that this $W$ is a basis for $-\Sigma$ also. To prove this, first of all we will establish that $W$ is a resolving set for $-\Sigma$. As $W$ is a resolving set for $\Sigma$, given $u,v\in V(\Sigma)$, there exists $i$ such that $1\le i\le k$ with $d_\Sigma(u,w_i)\neq d_\Sigma(v,w_i)$. As our motive is to prove that \begin{equation}\label{eq2}
d_{-\Sigma}(u,w_i)\neq d_{-\Sigma}(v,w_i)	
\end{equation} for the same $i$ we proceed with the following three cases.\\
\textbf{Case(i)}: Both the shortest paths $P(u,w_i)$ and $P(v,w_i)$ are of even length. In this case, using \eqref{eq1}, $\sigma'(P(u,w_i))=\sigma(P(u,w_i))$ and $\sigma'(P(v,w_i))=\sigma(P(v,w_i))$ so that $d_{-\Sigma}(u,w_i)=d_{\Sigma}(u,w_i)$ and $d_{-\Sigma}(v,w_i)=d_{\Sigma}(v,w_i)$ which proves the requisite condition in \eqref{eq2}.\\
\textbf{Case(ii)}: Both the shortest paths $P(u,w_i)$ and $P(v,w_i)$ are of odd length. Here, using \eqref{eq1}, $\sigma'(P(u,w_i))=-\sigma(P(u,w_i))$ and $\sigma'(P(v,w_i))=-\sigma(P(v,w_i))$ so that $d_{-\Sigma}(u,w_i)=-d_{\Sigma}(u,w_i)$ and $d_{-\Sigma}(v,w_i)=-d_{\Sigma}(v,w_i)$ which proves the condition in \eqref{eq2}.\\
\textbf{Case(iii)}: The shortest path $P(u,w_i)$ is of odd length and $P(v,w_i)$ is of even length. Then it is immediate that $d_{G}(u,w_i)\neq d_{G}(v,w_i)$ and as such \eqref{eq2} remains true.
Thus in all cases, the condition in \eqref{eq2} holds good  which proves the claim that $W$ is a resolving set for $-\Sigma$ also. To prove that it is a basis, if possible let $W'$ be a resolving set for $-\Sigma$ with $|W'|<|W|$, then it leads to a contradiction since $W'$ will then be a resolving set for $-(-\Sigma)=\Sigma$.
\end{proof}
Denoting an all-positive signed complete graph by $K_n^+$ and an all-negative signed complete graph by $K_n^-$, we have the following.
\begin{cor}\label{complete1}
	$\dim(K_n^+)=\dim (K_n^-)=n-1$.
\end{cor}
\begin{proof}
	Note that $K_n^-=-K_n^+$ and apply Theorem~\ref{negate}.
\end{proof}
\begin{cor}\label{metdimall}
The metric dimension of an all-negative signed cycle is $2$.
\end{cor}
\begin{proof}
If $C_n^\sigma=(C_n,\ \sigma)$ is the all negative signed cycle, then $-C_n^\sigma=C_n$ and the well known result \cite{chartrand} $\dim(C_n)=2$ completes the proof from Theorem~\ref{negate}.
\end{proof}
In the case of an unsigned complete graph $K_n$, it is well known that $\dim(K_n)=n-1$. It is not generally true in the case of signed complete graphs as illustrated in the following example, but we have the bounds for the metric dimension of signed complete graphs given in Theorem~\ref{complete}.
\begin{ex}\rm{Consider $K_3^{\sigma_1}$ where there is only one negative edge. Then using Theorem~\ref{oddcycle}, $\dim(K_3^{\sigma_1})=1$. Also for a complete signed graph with the underlying graph $K_4$ having only one negative edge, its metric dimension can be easily verified to be $2$. Moreover using the possible vectors $(0,\pm 1),\ (\pm 1,0), \ (1,1),\ (-1,1),\ (1,-1),\ (-1,-1)$, one can really construct signed complete graphs of order up to $6$ with the metric dimension $2$.}
\end{ex}
From the above discussion, we arrive at the following theorem in the case of signed complete graphs of order greater than $3$.
\begin{thm}\label{complete}
The metric dimension of a signed complete graph $K_n^\sigma=(K_n,\sigma)$ satisfies $2\le\dim(K_n^\sigma)\le n-1$ for all $n\ge 4$.
\end{thm}
\begin{thm}
The metric dimension of a signed path is $1$.
\end{thm}
\begin{proof}
Let $P_n^\sigma=(P_n,\sigma)$ be the given signed path where $P_n:u_1u_2\cdots u_n$. $P_n^\sigma$ being a compatible signed graph, letting $W=\{u_1\}$, it can be easily seen that $r(u_i|W)=(\sigma(u_1u_i)(i-1))$ for $i\neq 1$ and $(0)$ for $i=1$. This shows that $W$ is a resolving set for $P_n^\sigma$ and $\dim(P_n^\sigma)=1$.\\
Indeed, there is a shorter proof, if we use Theorem~\ref{gandsigma} as follows. $1\le\dim(P_n^\sigma)\le\dim(P_n)=1$.
\end{proof}
For an unsigned graph, the converse is also true and an unsigned path is characterized to be the only graph having metric dimension $1$. But in the case of signed graphs, it is not the case as proved in the following theorems. Note that in the case of an unsigned cycle $C_n,\ \dim(C_n)=2$. We denote by $d^+(v)$, the number of positive edges incident with the vertex $v$ in a signed graph and similarly $d^-(v)$ denotes the number of negative edges incident with $v$. The net-degree of a vertex $v$ in a signed graph is $d^\pm(v)=d^+(v)-d^-(v)$. Also, open $k$-neighbourhood of a vertex $v$ in a signed graph $\Sigma=(G,\sigma)$, denoted by  $N_k(v)$, is $N_k(v)=\{u\in V(\Sigma):d(u,v)=\pm\ k\}$.
\begin{thm}\label{oddcycle}
Let $C_n^\sigma=(C_n,\sigma)$ be a signed cycle. Then $\dim(C_n^\sigma)=1$ if and only if there exists a vertex $u$ with $d^\pm(u)=0$ and for $v,w\in N_k(u),\  \sigma(uv)\neq \sigma(uw)$ whenever $v\neq w$  for all $k=1,2,\cdots,\lfloor n/2 \rfloor$.  	
\end{thm}
\begin{proof} First we deal with the case when $n=2p+1$. The case for even cycle follow along the same lines. Let the underlying odd cycle $C_{2p+1}$ has the vertex set $V$  and let $u_1\in V$ be the vertex with $d^\pm(u_1)=0$. We claim that $W=\{u_1\}$ is a resolving set and hence a basis, if for $v,w\in N_k(u_1),\  \sigma(u_1v)\neq \sigma(u_1w)$ whenever $v\neq w$  for all $k=1,2,\cdots, p$. Clearly, if $v\in N_i(u_1)$ and $w\in N_j(u_1)$ for $i\ne j$, then $r(v|W)\neq r(w|W)$ as $d(v,u_1)\neq d(w,u_1)$. Therefore, we choose $v,w\in N_k(u_1)$ with $v\neq w$ where $k=1,2,\cdots, p$. Then by the assumption that $\sigma(u_1v)\neq \sigma(u_1w)$,    $d(v,u_1)\neq d(w,u_1)$ proving that $r(v|W)\neq r(w|W)$. This establishes one way implication. For the converse, first we note that existence of a vertex $u_1$ with net degree zero is essential and only such vertices can come in a basis of cardinality $1$.  Moreover, if for some $v,w\in N_k(u_1), \ \sigma(u_1v)=\sigma(u_1w)$, then $r(v|W)=r(w|W)$ where $W=\{u_1\}$, implying that, it is not a resolving set and hence not a basis.

In the case of an even signed cycle, using Theorem~\ref{compcycle}, it is compatible only when it is balanced. Hence to check its metric dimension is $1$, in view of Corollary~\ref{metdimall}, it is first of all necessary that there must exist a vertex of net degree $0$ and thereafter the remaining assumptions in the theorem.	
\end{proof}
Now we deal with the metric dimension of signed stars. In the case of all-positive and all-negative signed stars, the dimension, with the help of the established result \cite{chartrand} and Theorem~\ref{negate}, is $n-1$.
\begin{thm}\label{star}
	$\dim(K_{1,n}^\sigma)=n-2$ for $n\geq 3$ when it is non-homogeneous.
\end{thm}
\begin{proof}
Let $V=\{u\}\cup V_1 \cup V_2$ be the vertex set where $u$ is the central vertex of $K_{1,n}$ and $V_1=\{v_1,v_2,\cdots v_l\}$ and $ V_2=\{u_1,u_2,\cdots u_m\}$ be the vertices of net-degrees $1$ and $-1$ respectively with $l+m=n$. Since  $K_{1,n}^{\sigma}$ is non-homogeneous, $l\ge 1$ and $m\ge 1$. Define an ordered set $W = \{v_1,v_2,...v_{l-1},u_1,u_2,...u_{m-1}\}$. We claim that this is the required basis for  $K_{1,n}^\sigma$. First we prove that $W$ is a resolving set, for which it is enough to establish that metric representation of $u,v_l$ and $u_m$ are different. Denoting the $i$-th co-ordinate of $r(u|W)$ by $r(u|W)(i)$,
	\begin{align*} r(u|W)(i)&=1 \mbox{\ for\ } 1\le i\le l-1\\
		&=-1 \mbox{\ for\ } l\le i\le n-2
	\end{align*} and
	\begin{align*} r(v_l|W)(i)&=2 \mbox{\ for\ } 1\le i\le l-1\\
		&=-2 \mbox{\ for\ } l\le i\le n-2
	\end{align*}
	Similarly,
	\begin{align*} r(u_m|W)(i)&=-2 \mbox{\ for\ } 1\le i\le l-1\\
		&=2 \mbox{\ for\ } l\le i\le n-2
	\end{align*}	
Thus all the vertices have different metric representations with respect to $W$ proving that it is a resolving set. To show that $W$ is minimum, if possible let $W'$ be a resolving set with $|W'|<|W|$. Assume $|W'|=k<n-2$ and let $W'=\{x_1,x_2,\cdots x_h\}\cup \{x_{h+1},x_{h+2},\cdots, x_k\}=W_1\cup W_2$, say, such that vertices in $W_1$ are of net-degree $1$ and $W_2$ contains vertices of net-degree $-1$ and possibly, they may include the central vertex of the signed star too. Then, outside $W_1 \cup W_2$ there exist at least four vertices from among them we can pick at least two vertices $x_i$ and $x_j$ such that $\sigma(ux_i) \sigma(ux_j) = 1$, which implies that $r(x_i|W')= r(x_j|W')$, a contradiction to the fact that $W'$ is a resolving set. Hence the result.	
\end{proof}
\section{Metric dimension of signed wheels}\label{s2}
This section deals exclusively with the metric dimension of signed wheels and allied results on resolving sets for a signed wheel. A wheel $W_{n}$ is the join of a cycle $C_n$ and $K_1$. The vertex of $K_1$ is called the central vertex of the wheel $W_{n}$ and the vertices on the cycle is often referred to as its rim vertices. The following theorem, found in \cite{shan}, gives the metric dimension of an unsigned wheel.
\begin{thm}[\cite{shan}]\label{wheel}
	\begin{equation}\label{equdag}
	\dim(W_{n})= \left\{
	\begin{array}{rlll}  \Bigl\lfloor \dfrac{2n+2}{5}\Bigr\rfloor,& \text{if } n\ge 7,\\
	2, & \text{if } n=4,5,\\
	3, & \text{if } n=3,6.
	\end{array} \right.
	\end{equation}	
\end{thm}
We use the notation $W_{n}^\sigma=(W_{n},\sigma)$ to denote a signed wheel. For any two vertices $u$ and $v$ in $ W_n $, either $d(u,v)=1$ or $d(u,v)=2 $ as the diameter of any wheel is at most $2$. So to discuss the distance compatibility of $W_{n}^{\sigma} $, where $n\ge 4$, in view of Theorem~\ref{compcycle}, it is enough to consider only $ C_4 $ with one vertex as the central vertex of $W_{n}^{\sigma}$. We establish the criteria for the distance compatibility of signed wheels as follows in which $C_4^-$ denotes a negative cycle with four vertices.
  \begin{thm}
A signed wheel $W_{n}^{\sigma}$, where $n\ge 4$, is distance compatible if and only if it is $C_{4}^-$-free where one of the vertices of the cycle is the central vertex of  $W_{n}^{\sigma}$.
\end{thm}
\begin{proof}
Consider a distance compatible signed wheel $W_{n}^{\sigma}$ with vertex set $\{v, v_{1},v_{2},....v_{n} \}$, where $v$ is the central vertex of  $W_{n}^{\sigma}  $. First of all, assume that  $W_{n}^{\sigma}$ has a negative even cycle $C_4^-$ with $C_4 : vv_1v_2v_3v$. Then  $C_4$ has either one negative edge or three negative edges. In both the cases $\sigma_{max} (v_1 v_3)= 1 $ and $\sigma_{min} (v_1 v_3)= -1 $ and hence $d^{max} (v_1 ,v_3)\neq d^{min}(v_1, v_3) $, which shows $W_n^\sigma$ is not distance compatible, a contradiction.
\\For the converse, assume that $W_n^\sigma$ has no negative even cycle $C_4$ with one of the vertices on it as the central vertex of the wheel. Let $u$ and $v$ be any two vertices of $W_n$. Then the following two cases arise.
	\\  $\boldsymbol {Case$\space$  (i)}$: $u$ and $v$ are adjacent vertices.
		\\If $\sigma(uv)=1$, then $\sigma_{max} (uv) = 1= \sigma_{min} (uv)$ which makes $d^{max}(u,v)=d^{min}(u,v)=1$.
		\\If $\sigma(uv)=-1$, then $\sigma_{max} (uv) = -1= \sigma_{min} (uv)$ implying that $d^{max}(u,v)=d^{min}(u,v)=-1$.
	So $u$ and $v$ are distance compatible.
	\\ $\boldsymbol{Case $\space $ (ii)}$: $u$ and $v$ are nonadjacent vertices so that $d(u,v)=2$. Assume that $\sigma_{max}(uv) \neq \sigma_{min}(uv)$. Since the case $\sigma_{max}(uv)=-1$  and $\sigma_{min}(uv)= 1 $ is not possible simultaneously, we have $ \sigma_{max}(uv)=1 $ and $\sigma_{min}(uv)=-1$. So by definition, there exist a positive and  negative $u-v$ paths of length $2$. These two paths together makes a negative even cycle $C_4^-$ with one vertex as the central vertex of the wheel, a contradiction. Hence $\sigma_{max}(uv)= \sigma_{min}(uv)$ which implies $d^{max}(u,v)=d^{min}(u,v)$ and so $u $ and $v$ are distance compatible.	
\end{proof}
\begin{lem}\label{wheel1}A set of cardinality one cannot resolve a signed wheel $W_{n}^\sigma \ \ \forall \ n \geq 3$. In other words, $\dim(W_n^\sigma)\ge 2$.
\end{lem}
\begin{proof}$\boldsymbol{Case $\space $ (1)}$: Let $B=\{v_0\}$, where $v_0$ is the central vertex. Then there exists three rim vertices $v_1,v_2,v_3$ with $r(v_i|B)=1$. Since \ $\sigma(v_0v_1)=\sigma(v_0v_2)$ \ or \ $\sigma(v_0v_1)=\sigma(v_0v_3)$, this implies $\ r_\Sigma(v_1|B)=r_\Sigma(v_2|B)$ \ \ or \ \ $r_\Sigma(v_1|B)=r_\Sigma(v_3|B)$. Hence, $B = \{v_0\}$ \ does not resolve $W_{n}^\sigma$.\\
$\boldsymbol{Case $\space $ (2)}$:  Assume that $B\neq \{v_0\}$ and let $B=\{v_i\},  v_i\in V \backslash \{v_0\}$. Then \ $\sigma(v_{i-1}v_i)=\sigma(v_{i+1}v_i)$   or  $\sigma(v_{i-1}v_i)=\sigma(v_0v_i)$. Therefore, $\ r_{\Sigma}(v_{i-1}|B)=r_{\Sigma}(v_{i+1}|B)$ \ \ or \ \ $r_{\Sigma}(v_{i-1}|B)=r_{\Sigma}(v_0|B)$ showing that $B$ cannot resolve $W_{n}^\sigma$. Hence the proof.
\end{proof}
\begin{rmk}\rm{For a wheel $W_n, n>6$, let $\dim (W_n)=k$. To show that the central vertex of this unsigned wheel should not be a member of any of its basis, one can have the following argument. If we choose any set $B$ with $k-1$ rim vertices then there are two rim vertices $v_i,v_j$ such that $r(v_i|B)=r(v_j|B)=(x_1,x_2,\cdots,x_{k-1})$, say. So if $B_1=\{v_0\}\cup\ B$, where $v_0$ is the central vertex, then $r(v_i|B_1)=r(v_j|B_1)=(1,x_1,x_2,\cdots,x_{k-1})$. So $B_1$ does not resolve $W_n$ when $n>6$. This completes the argument for the above result. The same is the case with the signed wheel also as shown in the following theorem.}
\end{rmk}
\begin{thm} \label{v0}
For $n=3,4,5,6$, the central vertex may be chosen as an element of a basis of $W_{n}^\sigma$ and for $n>6$, the central vertex should not be an element of a basis of $W_{n}^\sigma$.
\end{thm}
\begin{proof}
For $n=3$, the underlying wheel (note that it is a complete graph $K_4$) has a basis $B$ which includes its central vertex. So this basis $B$ resolves $W_{n}^\sigma$ also and if we choose a set $B_{1}$ with two elements, which includes the central vertex $v_{0}$ and a rim vertex, without loss of generality, let it be $v_{1}$, then there exist two rim vertices with $r(v_{2}|B_{1})\ = r(v_{3}|B_{1})\ = (1,1)$. So if $\sigma(v_{0}v_{2}) \neq \sigma(v_{0}v_{3})$, then $r_{\Sigma}(v_{2}|B_{1})\ \neq r_{\Sigma}(v_{3}|B_{1})$. By Lemma \ref{wheel1}, as no singleton set resolves a signed wheel, $B_{1}$ will act as a basis of $W_{3}^\sigma$. On the other hand, if $\sigma(v_{0}v_{2}) = \sigma(v_{0}v_{3})$, then $B$ will becomes a basis of $W_{3}^\sigma$.\\
For $n=4$, $B_{1}=\{v_{0},v_{1}\}$, where  $v_0$ is the central vertex and $v_{1}$ is a rim vertex, does not resolve $W_{n}$, since we have $r(v_{2}|B_{1})\ = r(v_{4}|B_{1})\ = (1,1)$. So the central vertex cannot be an element of any basis of $W_{4}$. But if 
\begin{equation}\label{t1}
	\sigma(v_{0}v_{2}) \neq \sigma(v_{0}v_{4})
\end{equation}
, then  $r_{\Sigma}(v_{2}|B_{1})\ \neq r_{\Sigma}(v_{4}|B_{1})$. By Lemma \ref{wheel1}, this gives $B_{1}$ as basis of $W_{4}^\sigma$. But when, $\sigma(v_{0}v_{2}) = \sigma(v_{0}v_{4})$, $B_{1}$ does not resolve $W_{4}^\sigma$. So the central vertex cannot be an element of basis of $W_{4}^\sigma$. Hence if Equation~\eqref{t1} holds, by Lemma \ref{wheel1}, $W_4^{\sigma}$ can have a basis with central vertex.	\\
For $n=5$, $B_{1}=\{v_{0},v_{1}\}$, where  $v_0$ is the central vertex  and $v_{1}$ is a rim vertex does not resolve $W_{5}$, since we have $r(v_{2}|B_{1})\ = r(v_{5}|B_{1})\ = (1,1)$ and $r(v_{3}|B_{1})\ = r(v_{4}|B_{1})\ = (1,2)$. So the central vertex cannot be an element of basis of $W_{5}$. Now, if
\begin{equation}\label{t2}
	\sigma(v_{0}v_{2}) \neq \sigma(v_{0}v_{5})\ and\ \sigma(v_{0}v_{3}) \neq \sigma(v_{0}v_{4})
\end{equation}, then $r_{\Sigma}(v_{i}|B_{1})\ \neq r_{\Sigma}(v_{j}|B_{1})$, $\forall \ v_{i},   v_{j} \in V$. Therefore, by Lemma \ref{wheel1}, $B_{1}$ acts as a basis of $W_{5}^\sigma$. Moreover, if  $\sigma(v_{0}v_{2}) = \sigma(v_{0}v_{5})$ or $\sigma(v_{0}v_{3}) = \sigma(v_{0}v_{4})$, then $B_{1}$ cannot resolve $W_{5}^\sigma$. Hence if Equation~\eqref{t2} holds, by Lemma \ref{wheel1}, $W_5^{\sigma}$ can have a basis with the central vertex.\\
For $n=6$, the underlying graph has a basis $W$ which includes the central vertex, which resolves $W_{6}^\sigma$ also and if we choose a set $B_{1}$ with two vertices, which includes $v_{0}$ and an arbitrary rim vertex $v_{1}$, say, then there will be three vertices with same metric representation in the underlying graph and let it be $v_{a},v_{b},v_{c}$. As such,
		\begin{equation} \label{w}
			 \sigma(v_{0}v_{a})= \sigma(v_{0}v_{b})\  or\  \sigma (v_{0}v_{a})= \sigma(v_{0}v_{c}).
		\end{equation}
Therefore, $r_{\Sigma}(v_{a}|B_{1})\ = r_{\Sigma}(v_{b}|B_{1})$ or $r_{\Sigma}(v_{a}|B_{1})\ = r_{\Sigma}(v_{b}|B_{1})$. This shows that $ B_{1}$ does not resolve $W_{n}^{\sigma}$. Now, a set with 3 rim vertices resolves $W_{6}^{\sigma}$ and let $ B_{1}=\{v_{k1},v_{k2},v_{k3}\}\subseteq V\setminus\{v_0\}$  be a set with $ r(v_{i}|B_{1})\neq  \ r(v_{j}|B_{1}) ;  \forall  \ v_{i}, v_{j} \in V $. By Theorem \ref{undres}, this gives $\ r_\Sigma(v_{i}|B_{1})\neq  \ r_\Sigma(v_{j}|B_{1}) \ \forall  \ v_{i}, v_{j} \in V $. We proceed with the following cases and subcases.	\\
\underline{Case - A} Consider any set with two vertices without the central vertex, let it be $ B_{2}=\{v_{1},v_{4}\}\subseteq V $ or $B_{2}=\{v_{1},v_{2}\}\subseteq V $ or $B_{2}=\{v_{1},v_{6}\}\subseteq V$. Then, $ r(v_{p}|B_{2})= \ r(v_{q}|B_{2})$ for $v_{p},v_{q} \in V\setminus\{v_{0}\}$. Let $S=\{(v_{p}, v_{q})|\sigma(v_{0}v_{p})=\sigma(v_{0}v_{q})\}$. \\
\underline{Subcase-1:} If $ (v_{p},v_{q})\in S$ then $\ r_\Sigma(v_{p}|B_{2})= r_\Sigma(v_{q}|B_{2}) $. Therefore, $B_{2}$ does not resolve $W_{6}^{\sigma}$\\
\underline{Subcase-2:} If all such $ (v_{p},v_{q})\notin S$ then $\ r_\Sigma(v_{p}|B_{2})\neq r_\Sigma(v_{q}|B_{2}) $. So, $B_{2} $ resolves $W_{6}^{\sigma}$.\\
\underline{Case - B} If $B_{2}=\{v_{1},v_{3}\}$ or $\{v_{1},v_{5}\}$ then $ r(v_{0}|B_{2})= \ r(v_{p}|B_{2})=(1,1)$ for a $v_{p}\in V\setminus\{v_{0}\}$
		\\ 1) If $\ r_\Sigma(v_{0}|B_{2})=(1,1)$ and $ r_\Sigma(v_{p}|B_{2})\in\{(1,-1), (-1,1), (-1,-1)\}; \ v_{p}\in V\setminus B_{2}$ where $p\neq0$, $B_{2}$ resolves $W_{6}^{\sigma}$.
		\\
		2) If $\ r_\Sigma(v_{0}|B_{2})=(-1,-1)$ and $ r_\Sigma(v_{p}|B_{2})\in\{(1,-1), (-1,1), (1,1)\}; \ v_{p}\in V\setminus B_{2}$ where $p\neq0$, $B_{2}$ resolves $W_{6}^{\sigma}$.
		\\
		3) If $\ r_\Sigma(v_{0}|B_{2})=(1,-1)$ and $ r_\Sigma(v_{p}|B_{2})\in\{(-1,1), (1,1), (-1,-1)\};v_{p}\in V\setminus B_{2}$ where $p\neq0$, $B_{2}$ resolves $W_{6}^{\sigma}$.
		\\
		4)  If $\ r_\Sigma(v_{0}|B_{2})=(-1,1)$ and $ r_\Sigma(v_{p}|B_{2})\in\{(1,1), (1,-1), (-1,-1)\}\ ; v_{p}\in V\setminus B_{2}$ where $p\neq0$, $B_{2}$ resolves $W_{6}^{\sigma}$. Otherwise $B_{2} $ does not resolve $W_{6}^{\sigma}$.
		\\
		By Lemma \ref{wheel1}, no singleton set resolves the signed wheel and by Theorem \ref{gandsigma}, $\dim(W_{6}^{\sigma})=2$ [By Subcase 2 and Case B]. If the Subcase 2 and Case B do not hold then $\dim (W_{6}^{\sigma})=3$, so the central vertex can be an element of a basis of $W_{6}^{\sigma}$.
		\\
		For $n>6$, if $\dim(W_{n})=k$, then the underlying graph does not include the central vertex in its basis. So if we take $B=\{v_{0},v_{m_{1}},v_{m_{2}},...v_{m_{k-1}}\}$, then there exist $(v_{p},v_{q}) $ such that $ r(v_{p}|B)\ = r(v_{q}|B)$, for some $v_{p},v_{q}\in V$.
		\\
		If $\sigma(v_{0}v_{p}) \neq \sigma(v_{0}v_{q})$ for  $(v_{p},v_{q})$ satisfying $r(v_{p}|B)\ = r(v_{q}|B)$, then 
		$r_{\Sigma}(v_{i}|B)\ \neq r_{\Sigma}(v_{j}|B)$, $\forall \ v_{i}, v_{j} \in V$ and as such, $B$ resolves $W_{n}^\sigma$. But if we choose $B_{1}=B\setminus \{v_{0}\}$ in  graph satisfying the above condition that $\sigma(v_{0}v_{p}) \neq \sigma(v_{0}v_{q})$ for  $(v_{p},v_{q})$ where $r(v_{p}|B)\ = r(v_{q}|B)$, then $B_{1}$ resolves  $W_{n}^\sigma$. This shows that no basis of $W_{n}^\sigma$ can include the central vertex for $n>6$.
\end{proof}
The following theorem gives the metric dimensional difference $\mdd(W_{n}^\sigma)$.
 \begin{thm}\label{wl}
Let $W_{n}^\sigma=(W_{n},\sigma)$ be a compatible signed wheel where $n\ge 3$. Then $\mdd(W_{n}^\sigma)\le 1$.
\end{thm}
\begin{proof} For simplicity of notations, we use $G$ for denoting the underlying wheel and $\Sigma$ for the signed wheel and deal with the following cases depending on the values of $n$.
\\ \textbf{Case-1:} [ n=3 ] \\
Let V=$\{v_{0},v_{1},v_{2},v_{3}\} $\ be the vertex set. Then $\dim (G) =3$ from Theorem~\ref{wheel}. By Theorem~\ref{undres}, there exist a set with $3$ vertices which resolves $\Sigma$ also. Let $\{v_{a},v_{b},v_{c}\} $\ $\subseteq V $ resolve  $\Sigma$. If we take any set with two vertices, let it be $B=\{v_{a},v_{b}\}$\, then easy computations give the following values as the case may be.  
\\$\ r(v_{c}|B)=r(v_{d}|B)=(1,1)$ where $v_{c},v_{d} \in V$. Therefore, if $\ r_\Sigma(v_{c}|B)=(1,1)$\ and  $\ r_\Sigma(v_{d}|B) \in \{(1,-1) ,(-1,1) , (-1,-1)\} $ or $\ r_\Sigma(v_{c}|B)=(-1,-1)$\ and  $\ r_\Sigma(v_{d}|B)\in \{(1,-1)  ,(-1,1), (1,1)\}$  or $ r_\Sigma(v_{c}|B)=(-1,1)$\ and  $\ r_\Sigma(v_{d}|B) \in \{(1,1),(-1,-1),(1,-1)\}$ or $\ r_\Sigma(v_{c}|B)=(1,-1)$\ and  $\ r_\Sigma(v_{d}|B) \in \{(1,1), (-1,-1),(-1,1)\}$. These different sets of values provide
 \begin{equation} \label{eqq}
 	  r_\Sigma(v_{c}|B)\neq  \ r_\Sigma(v_{d}|B)
 \end{equation} making $\{v_{a},v_{b}\}$\ as a resolving set for $\Sigma$. By Lemma~\ref{wheel1},  as a singleton set does not resolve $\Sigma$, $\{v_{a} , v_{b}\}$\ acts as basis. If the above condition in \eqref{eqq} does not hold, then by Theorem \ref{gandsigma} $\dim(\Sigma)=3$. This proves that $\dim(W_{3}^\sigma)$ is either $2$ or $3$.
\\  \textbf{Case-2:}  [$n = 4,5$ ] \\
Here, $\dim (G) =2$. Therefore, there exist $B=\{v_{p},v_{q}\}\subseteq V$\  such that $\ r(v_{i}|B)\neq  \ r(v_{j}|B) ; \forall \ v_{i}, v_{j} \in V $. By Theorem \ref{undres} and Theorem \ref{v0}, this means $\ r_\Sigma(v_{i}|B)\neq  \ r_\Sigma(v_{j}|B) \ \forall  \ v_{i}, v_{j} \in V $. Also by Lemma \ref{wheel1}, singleton set does not resolve $\Sigma$ gives $\dim(W_{n}^\sigma)=2$ when $n=4,5$.
\\
\textbf{Case-3:}  [ $n=6$ ] \\
In this case, $\dim (G)=3$ and  letting $V=\{v_{0},v_{1},v_{2}....v_{6}\}$, there exists a basis $B_{1}=\{v_{k1},v_{k2},v_{k3}\}\subseteq V $  such that $ r(v_{i}|B_{1})\neq  \ r(v_{j}|B_{1}) ; $ for all $ v_{i},  v_{j} \in V $. So by Theorem \ref{undres}, $\ r_\Sigma(v_{i}|B_{1})\neq  \ r_\Sigma(v_{j}|B_{1}) \ \forall  \ v_{i}, v_{j} \in V $\
\\
If Subcase 2 of Case A and Case B of Theorem \ref{v0} holds, then $\dim(\Sigma)=2$. Otherwise $\dim(\Sigma) =3$. This proves that, $\dim(W_{6}^\sigma)$ is either $2$ or $3$.
\\
\textbf{Case-4} [ $n > 6$ ] \\
Let $V=\{v_{0},v_{1},v_{2},...v_{n}\}$ where $v_{0}$ is the central vertex and $\{v_{1},v_{2},...v_{n}\}$ are the rim vertices. If $\dim (G) =k$, then let $B_{1}=\{v_{i1},v_{i2},...v_{ik}\}\subseteq V\setminus\{v_{0}\}$ resolve $G$. By Theorem \ref{undres}, $B_{1}$ resolves $\Sigma$ too. If we take any set with $k-1$ vertices, $B_{2}=\{v_{j1},v_{j2},...v_{jk-1}\}\subseteq V\setminus\{v_{0}\}$ then $ r(v_{s}|B_{2})= \ r(v_{t}|B_{2})$ for some  $v_{s},v_{t}\in V$. Define $S=\{(v_{s},v_{t}) \| \  r(v_{s}|B_{2})= \ r(v_{t}|B_{2})\}$ and partition S into two subsets $S_{1}$ and $S_{2}$ such that  $S_{1}=\{(v_{s},v_{t})\in S|\sigma(v_{0}v_{s})=\sigma(v_{0}v_{t})\}$ and $S_{2}=\{(v_{s},v_{t})\in S|\sigma(v_{0}v_{s})\neq\sigma(v_{0}v_{t})\}$.
  \\ 
  \underline{Subcase-A} \\
If at least one $(v_{s},v_{t})\in S$ belongs to $S_{1}$ then $\ r_\Sigma(v_{s}|B_{2})=r_\Sigma(v_{t}|B_{2}) $. Therefore, $B_{2}$ does not resolve $\Sigma$ and along the same line of argument, it can be shown that no other set with cardinality less than $k-1$ resolves $\Sigma$. By theorem \ref{gandsigma} this proves $\dim(\Sigma)=k$.
   \\
   \underline{Subcase-B} \\
If $S=S_{2}$, then $ r_\Sigma(v_{s}|B_{2})\neq r_\Sigma(v_{t}|B_{2}) $  $\  \forall  (v_{s},v_{t})\in S$ provides $B_{2}$ to resolve $\Sigma$. To settle the minimality, if we take any set of $k-2$ vertices, $W_{3}=\{v_{l1},v_{l2},...v_{lk-2}\}\subseteq V\ \backslash \{v_{0}\}$ then $ r(v_{a}|B_{3})=  r(v_{b}|B_{3})=\underbrace{(2,2....2)}_{(k-2) times}$ for some $v_{a}$ and $v_{b} \in V$. Define $S_{3}=\{(v_{a},v_{b})\in V: r(v_{a}|B_{3})=  r(v_{b}|B_{3})\}$. Since $n>6, |S_{3}|>2$, and so there exist at least three vertices $v_{a},v_{b},v_{c}\in V$ such that  $ r(v_{a}|B_{3})=  r(v_{b}|B_{3})=r(v_{c}|W_{3})=\underbrace{(2,2....2)}_{(k-2) times}$. This gives, $ \sigma(v_{0}v_{a})=\sigma(v_{0}v_{b})$ or $ \sigma(v_{0}v_{a})=\sigma(v_{0}v_{c})$. So, $r_\Sigma(v_{a}|B_{3})= r_\Sigma(v_{b}|B_{3}) $ or $ r_\Sigma(v_{a}|B_{3})= r_\Sigma(v_{c}|B_{3})$. Hence, $B_{3}$ does not resolve $\Sigma$ and proceeding along the same lines, no set with cardinality less than $k-2$ resolves $\Sigma$. By Theorem \ref{gandsigma},  $\dim(\Sigma)=k$ or $k-1$.      
Combining all the cases, $\dim (W_{n}^\sigma)= \dim(W_n) $ or $\dim(W_n)-1$. In other words, $\mdd(W_{n}^\sigma)=0$ or $1$, which proves the theorem.
  \end{proof}
The figures that follow illustrate the case when $n=9$. The thick lines in all the figures represent positive edges and the dashed lines are negative edges. The blackened vertices belong to the basis in the order of their suffixes.
\begin{figure}[h!]
	\centering
	\begin{tikzpicture}
	\draw[thick](0,0)--(0,3)--(2.1,2.3)--(3,.7)--(2.8,-1.4)(1.2,-2.8)--(-1.2,-2.8)--(-2.8,-1.4)(-3,.7)--(-2.1,2.3)--(0,3)(0,0)--(2.1,2.3)(0,0)--(3,.7)(0,0)--(3,.7)(0,0)--(2.8,-1.4)(0,0)--(-3,.7)(0,0)--(-2.1,2.3);
	\draw[dashed](0,0)--(1.2,-2.8)(0,0)--(-1.2,-2.8)(0,0)--(-2.8,-1.4)(-2.8,-1.4)--(-3,.7)(2.8,-1.4)--(1.2,-2.8);
	\draw[fill=white](0,0)circle(3.5pt); 
	\draw[fill=black](0,3)circle (3.5pt); 
	\draw[fill=white](2.1,2.3)circle (3.5pt);
          \draw[fill=black](3,.7)circle(3.5pt);
          \draw[fill=white](2.8,-1.4)circle (3.5pt);
          \draw[fill=white](1.2,-2.8)circle (3.5pt); 
	\draw[fill=black](-1.2,-2.8)circle (3.5pt);
           \draw[fill=white](-2.8,-1.4)circle (3.5pt);
          \draw[fill=black](-3,.7)circle (3.5pt); 
	\draw[fill=white](-2.1,2.3)circle (3.5pt);
	\node at(0,-0.465){$v_0$};
	\node at(0.2,3.3){$v_1$};
	\node at(3.2,2.7){$v_2(1,1,-2,2)$};
	\node at(3.4,.7){$v_3$};
	\node at(4.2,-1.2){$v_4(2,1,-2,2)$};
	\node at(2.7,-3.1){$v_5(-2,-2,1,-2)$};
	\node at(-1.5,-3.1){$v_6$};
	\node at(-4.3,-1.7){$(-2,-2,1,-1)v_7$};
	\node at(-3.4,.7){$v_8$};
	\node at(-3.6,2.4){$(1,2,-2,1)v_9$};
	\end{tikzpicture}
\caption{$\dim(W_9^\sigma)=\dim(W_9)=4$, i.e., $\mdd(W_9^\sigma)=0$}
\label{fig1}
\end{figure}
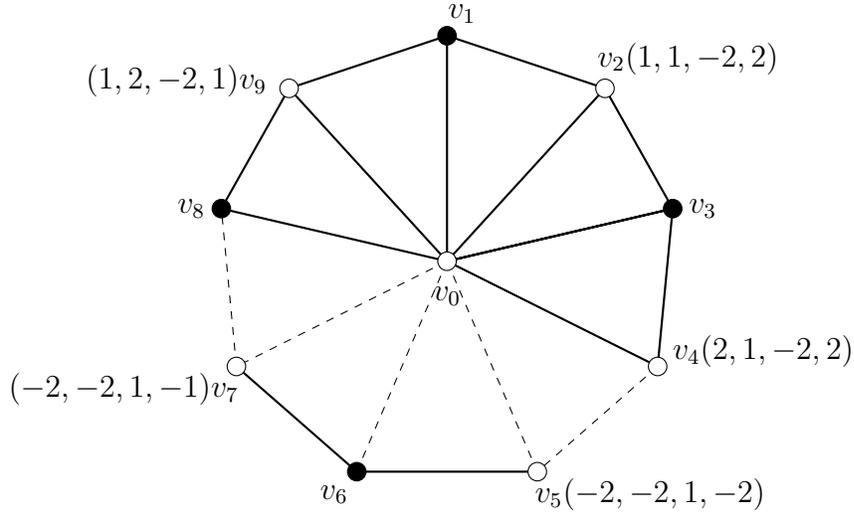
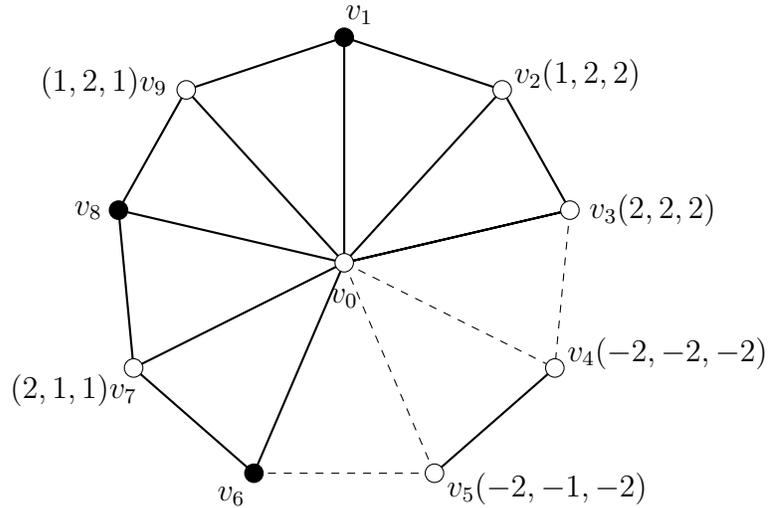
\begin{figure}[h!]
	\centering
	\begin{tikzpicture}
	
	\draw[thick](0,0)--(0,3)--(2.1,2.3)--(3,.7)(2.8,-1.4)(1.2,-2.8)(-1.2,-2.8)--(-2.8,-1.4)(-3,.7)--(-2.1,2.3)--(0,3)(0,0)--(2.1,2.3)(0,0)--(3,.7)(0,0)--(3,.7)(0,0)--(-3,.7)(0,0)--(-2.1,2.3)(0,0)--(-2.8,-1.4)(0,0)--(-1.2,-2.8)(2.8,-1.4)--(1.2,-2.8)(-2.8,-1.4)--(-3,.7);
	\draw[dashed](0,0)--(1.2,-2.8)(0,0)--(2.8,-1.4)(3,.7)--(2.8,-1.4)(1.2,-2.8)--(-1.2,-2.8);
	\draw[fill=white](0,0)circle(3.5pt); 
	\draw[fill=black](0,3)circle (3.5pt); 
	\draw[fill=white](2.1,2.3)circle (3.5pt);
          \draw[fill=white](3,.7)circle(3.5pt);
          \draw[fill=white](2.8,-1.4)circle (3.5pt);
          \draw[fill=white](1.2,-2.8)circle (3.5pt); 
	\draw[fill=black](-1.2,-2.8)circle (3.5pt);
           \draw[fill=white](-2.8,-1.4)circle (3.5pt);
          \draw[fill=black](-3,.7)circle (3.5pt); 
	\draw[fill=white](-2.1,2.3)circle (3.5pt);
	
	\node at(0,-0.4672){$v_0$};
	\node at(0.2,3.3){$v_1$};
	\node at(3.1,2.5){$v_2(1,2,2)$};
	\node at(4.1,.7){$v_3(2,2,2)$};
	\node at(4.3,-1.2){$v_4(-2,-2,-2)$};
	\node at(2.7,-3){$v_5(-2,-1,-2)$};
	\node at(-1.5,-3.1){$v_6$};
	\node at(-3.6,-1.7){$(2,1,1)v_7$};
	\node at(-3.4,.7){$v_8$};
	\node at(-3.2,2.4){$(1,2,1)v_9$};	
	\end{tikzpicture}
\caption{$\dim(W_9^\sigma)=3\neq \dim(W_9)=4$ \ \ i.e., $\mdd(W_9^\sigma)=1$}
\label{fig2}
\end{figure}

\section{Metric dimension of signed trees}\label{s3}
We now turn our attention to find the metric dimension of signed trees. Though we do not get a result for a general signed tree, Theorem~\ref{stree} deals with signed trees with some specific conditions. In the case of unsigned trees, there is a formula for finding the metric dimension given in ~\cite{chartrand, slater} which is stated in Theorem~\ref{tree}. But before delving into the details we shall recall some important definitions.
\\A \textit{leaf} of a tree is a vertex of degree $1$. A vertex of degree at least $3$ in a tree $T$ is called a \textit{major\ vertex}. A leaf $u$ of a tree $T$ is said to be a \textit{terminal\ vertex} of a major vertex $v$ of $T$ if $d(u,v)<d(u,w)$ for every other vertex $w$ of $T$. The \textit{terminal\ degree} $\ter(v)$ of a major vertex $v$ is the number of terminal vertices of $v$. We denote the totality of all terminal degrees in a tree $T$ by $\lambda(T)$. A major vertex $v$ of $T$  is an \textit{exterior\ major\ vertex} of $T$ if it has terminal degree greater than 0. Number of exterior major vertices is denoted by $\ext(T)$. A \textit{leg} of a tree is a path from an exterior major vertex to its terminal vertex. Let $t_k$ be an exterior major vertex of $T^{\sigma}$. By $L_{k,j}$, we mean the set of all vertices on the $j^{th}$ leg of $t_k$ and by the index $ i $ of a vertex $u_i \in L_{k,j}$, we mean $ d(t_k,u_i)=i$.
\begin{thm}[\cite{chartrand, slater}]\label{tree}
 If $ T $ is a tree, which is not a path, then $ \dim(T)=\lambda(T)-\ext(T) $.
\end{thm}
To deal with a special class of signed trees for which we derive a formula for their metric dimension,  we define one more type of exterior major vertices as follows.
\begin{defn}(Special exterior major vertex of a signed tree)
An exterior major vertex $t_{k}$ in a signed tree $T^\sigma$ is called a special exterior major vertex if  $d_{\Sigma }(t_k,u_i)\neq d_{\Sigma }(t_{k},v_i)$  for all $u_i \in L_{k,n}$ and $v_i \in L_{k,m} $, for $n \neq m$. The set of all special exterior major vertices is denoted as $\eta(T^\sigma)$.
\end{defn}
\begin{thm}\label{stree}
	Let $T^\sigma$ be a signed tree (which is not a path) such that $\ter(t_k) \neq 2 $ for all $ t_k \in \eta(T^{\sigma}) $. Then $\dim(T^\sigma)=\dim(T)-|\eta (T^\sigma)|$. In other words, in the case of such trees $T^{\sigma}$, $\mdd(T^{\sigma})=|\eta (T^\sigma)|$.
\end{thm}
\begin{proof} Let $\{t_{1},t_{2},...t_{n}\}$ be the exterior major vertices of $T^{\sigma}$ and let $ W_i $ denote be set of the terminal vertices $\{t_{i1},t_{i2},...t_{im_{i}}\}$ from $t_{i}$ for $1\le i\le n$. Denote $W = \cup W_i$. By Theorem~\ref{tree}, $\dim(T)=\lambda (T)-\ext(T)$ and the set formed from $W$ by removing  one terminal vertex  corresponding to each $t_{i}$, is a basis for $T$, and by Theorem \ref{undres}, it resolves $T^\sigma$ also. Now we modify this resolving set  to a basis of $T^{\sigma}$ by removing terminal vertices in a specific way as follows.\\Choose an exterior major vertex $t_k$ of $T^{\sigma}$. Then the following cases arise.
	\\ $\boldsymbol{Case (i)}$: $t_{k} \in \eta(T^{\sigma})$. In this case, there exist two terminal vertices $t_{kn}$ and  $t_{km}$ of  $t_{k}$ such that
	$d_{\Sigma }(u_{i},t_{k})\neq d_{\Sigma }(v_{i},t_{k})$  $\forall $ $ u_{i}\in L_{kn}$  and  $ \forall \ v_{i} \in L_{km}$. Therefore, we can remove $t_{kn} $ and $   t_{km} $  from $ W $. To prove the minimality, if we remove one more terminal vertex $t_{kl}$, say, of $t_{k}$ then 
	$d_{\Sigma }(t_{k},u_{1})= d_{\Sigma }(t_{k},v_{1})$  for $ u_{1}\in L_{kl}$  and  $  v_{1} \in L_{kn}$ or $L_{km}$. So we can remove two terminal vertices for a  $t_k \in  \eta(T^{\sigma}$) and we fix them as $t_{km} $ and $t_{kn}$.
	\\ $\boldsymbol{Case (ii)}$: $t_k \notin \eta(T^{\sigma})$. Here, there exist two legs $L_{kn}$ and $L_{km}$ of $t_k$  such that $ d_{\Sigma }(u_{i},t_{k}) =d_{\Sigma }(v_{i},t_{k})$  for some  $ u_{i}\in L_{kn}$  and  $ v_{i} \in L_{km}$. Hence we can remove at most one terminal vertex if $t_k \notin \eta(T^{\sigma}$) and choose it arbitrarily. Thus we can conclude that if $W^*$ is the set formed from $W$ by removing  two  terminal vertices as in Case (i), if $t_k \in \eta(T^{\sigma}) $ and remove one terminal vertex arbitrarily if $t_k \notin \eta(T^{\sigma})$, forms a basis of $T^{\sigma}$. Hence,   
  $$\dim(T^{\sigma})=|W^*|= \lambda(T)-\ext(T)-|\eta(T^{\sigma})|=\dim(T)-|\eta(T^{\sigma})|$$.	
\end{proof}
We have $|\eta(K_{1,n}^{\sigma})|=1$, if $ K_{1,n}^{\sigma}$ is non-homogeneous. So Theorem \ref{star} becomes a particular case of Theorem \ref{stree}. We provided a detailed proof there to show how one could select the resolving set of a signed star.
\\Generally, the following bounds hold good for any signed tree.
\begin{thm}
	Let $T^\sigma$ be a signed tree, then $\dim(T)-\ext(T)\leq \dim(T^{\sigma})\leq \dim(T)$. In other words, $\mdd(T^\sigma)\le \ext(T)$.
\end{thm}
\begin{proof}
From Theorem \ref{gandsigma},		
	\begin{equation} \label{eqa 1}
		\dim(T^{\sigma}) \leq\ \dim(T) 
	\end{equation}
As in the proof of Theorem~\ref{stree}, we can remove at most two terminal vertices with respect to each one of the exterior major vertices. Therefore, $\  T^{\sigma}$ has at least $\lambda(T)-2\ext(T)$ elements in its basis. So,
	\begin{equation} \label{eq}
	 \lambda(T)-2\ext(T) \leq \dim(T^{\sigma})
	\end{equation} Therefore, combining results in \eqref{eqa 1} and \eqref{eq}, $\dim(T)-\ext(T)\leq \dim(T^{\sigma})\leq \dim(T)$.
\end{proof}
The signed trees, one case of which is exhibited in the following figure, form an interesting class with the metric dimension as small as $2$ where as the unsinged underlying tree has a much greater metric dimension. 

\begin{figure}[h]
		\centering
	
		\begin{tikzpicture}
		\node at(1.2,-.3){$u$};
		\node at(13.2,-.3){$v$};
		
		\draw[thick](0,1.5)--(2,1.5)--(4,1.5)--(6,1.5)--(8,1.5)--(10,1.5)--(12,1.5)
		(-0.1,1.5)--(1,3)(4,1.5)--(5,3)(6,1.5)--(7,3)(8,1.5)--(9,3)(10,1.5)--(11,3)(2,1.5)--(3,0)(12,1.5)--(13,0);
		\draw[dashed](12,1.5)--(13,3)(2,1.5)--(3,3)(0,1.5)--(1,0)(4,1.5)--(5,0)(6,1.5)--(7,0)(8,1.5)--(9,0)(10,1.5)--(11,0);
		
		\draw[fill=white](-0.1,1.5)circle(3pt);
	    \draw[fill=white](2,1.5)circle (3pt);
		\draw[fill=white](4,1.5)circle (3pt);	
		\draw[fill=white](6,1.5)circle (3pt);	
		\draw[fill=white](8,1.5)circle (3pt);	
		\draw[fill=white](10,1.5)circle (3pt);
		\draw[fill=white](12,1.5)circle (3pt);	
		\draw[fill=white](1,3)circle (3pt);
		\draw[fill=black](1,0)circle (3pt);	
		\draw[fill=white](3,3)circle (3pt);	
		\draw[fill=white](3,0)circle (3pt);	
		\draw[fill=white](5,3)circle (3pt);
		\draw[fill=white](5,0)circle (3pt);	
		\draw[fill=white](7,3)circle (3pt);	
		\draw[fill=white](7,0)circle (3pt);	
		\draw[fill=white](9,3)circle (3pt);	
		\draw[fill=white](9,0)circle (3pt);	
		\draw[fill=white](11,3)circle (3pt);
		\draw[fill=white](11,0)circle (3pt);
		\draw[fill=white](13,3)circle (3pt);
		\draw[fill=black](13,0)circle (3pt);
	\end{tikzpicture}\\
		\hspace{3 cm}		
		\caption{ $\dim(T^\sigma)=2;\ \  \dim(T)=7$ i.e., $\mdd(T^\sigma)=5$}
		\label{fig3}
	\end{figure}
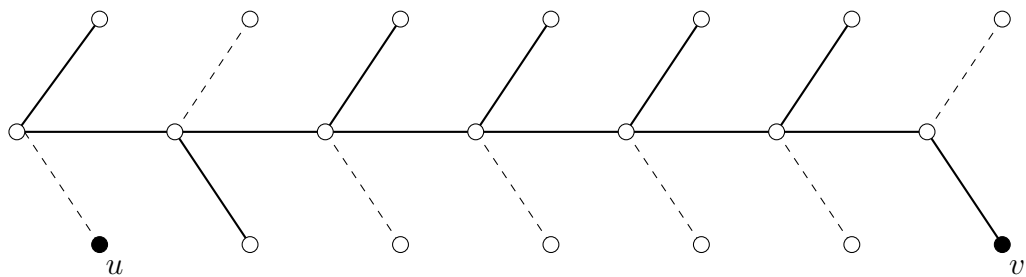


\end{document}